\documentclass[graybox]{svmult}

\usepackage{mathptmx}       
\usepackage{helvet}         
\usepackage{courier}        
\usepackage{type1cm}        
\usepackage{makeidx}         
\usepackage{graphicx}        
\usepackage{multicol}        

\usepackage{amssymb,amsfonts,amscd,amsmath}

\usepackage[bottom]{footmisc}

\makeindex             

\begin{document}

\title*{Lerch Quotients, Lerch Primes,\\Fermat-Wilson Quotients, and the Wieferich-non-Wilson Primes 2, 3, 14771}\titlerunning{Lerch Quotients and Primes, Fermat-Wilson quotients, and the WW Primes 2, 3, 14771} 
\author{{\bf Jonathan Sondow}\\209 West 97th Street, New York, NY 10025\\\email{jsondow@alumni.princeton.edu}}
\authorrunning{Jonathan Sondow}

%

\maketitle

\abstract{The Fermat quotient \mbox{$q_p(a):=(a^{p-1}-1)/p$}, for prime $p\nmid a$, and the Wilson quotient \mbox{$w_p:=((p-1)!+1)/p$} are integers. If $p\mid w_p,$ then $p$ is a Wilson prime. For odd~$p,$ Lerch proved that $(\sum_{a=1}^{p-1} q_p(a) - w_p)/p$ is also an integer; we call it the \emph{Lerch quotient}~$\ell_p.$ If $p\mid\ell_p$ we say $p$ is a \emph{Lerch prime}. A simple Bernoulli-number test for Lerch primes is proven. There are four Lerch primes $3, 103, 839,2237$ up to $3\times10^6$; we relate them to the known Wilson primes $5,13,563.$ Generalizations are suggested. Next, if $p$ is a non-Wilson prime, then $q_p(w_p)$ is an integer that we call the \emph{Fermat-Wilson quotient} of~$p.$ The GCD of all $q_p(w_p)$ is shown to be~24. If $p\mid q_p(a),$ then $p$ is a Wieferich prime base~$a$; we give a survey of them. Taking $a=w_p,$ if $p\mid q_p(w_p)$ we say $p$ is a \emph{Wieferich-non-Wilson prime}. There are three up to $10^7$, namely, $2,3,14771.$ Several open problems are discussed.
}

\section{Introduction}
\label{sec:intro}

By Fermat's little theorem and Wilson's theorem, if $p$ is a prime and $a$ is an integer not divisible by $p,$ then the \emph{Fermat quotient of $p$ base $a,$}
\begin{align}
	q_p(a) &:= \frac{a^{p-1}-1}{p},  \label{EQ: qDef}
\end{align}
and the \emph{Wilson quotient of $p,$}
\begin{align}
 	w_p &:= \frac{(p-1)!+1}{p},  \label{EQ: wDef}
\end{align}
are integers. (See \cite[pp.~16 and~19]{rib89} and \cite[pp. 216--217]{rib00}.)

For example, the Fermat quotients of the prime $p=5$ base $a=1,2,3,4$ are $q_5(a)=0,3,16,51$; the Fermat quotients of $p=3,5,7,11,13,17,19,23,29,31,\dotso$ base \mbox{$a=2$} are$$q_p(2)= \frac{2^{p-1}-1}{p}=1, 3, 9, 93, 315, 3855, 13797, 182361, 9256395, 34636833,\dotso$$\cite[sequence A007663]{oeis}; and the Wilson quotients of $p=2,3,5,7,11,13,17,\dotso$ are$$w_p=1, 1, 5, 103, 329891, 36846277, 1230752346353, \dotso$$\cite[sequence A007619]{oeis}.

A prime $p$ is called a \emph{Wilson prime} \cite[section A2]{upint}, \cite[p.~277]{rib89} if $p$ divides $w_p,$ that is, if the supercongruence$$(p-1)! + 1\equiv 0 \pmod{p^2}$$holds. (A \emph{supercongruence} is a congruence whose modulus is a prime power.) 

For $p=2,3,5,7,11,13,$ we find that $w_p \equiv 1, 1, 0, 5, 1, 0 \pmod {p}$ (see \cite[sequence A002068]{oeis}), and so the first two Wilson primes are $5$~and~$13.$ The third and largest known one is~$563,$ uncovered by Goldberg~\cite{goldberg} in $1953.$ Crandall, Dilcher, and Pomerance~\cite{cdp} reported in $1997$ that there are no new Wilson primes up to $5\times 10^{8}$. The bound was raised to $2\times 10^{13}$ by Costa, Gerbicz, and Harvey \cite{cgh} in $2012$.

Vandiver in $1955$ famously said (as quoted by MacHale~\cite[p.~140]{machale}):
\begin{quote}
\noindent It is not known if there are infinitely many Wilson primes. This question seems to be of such a character that if I~should come to life any time after my death and some mathematician were to tell me that it had definitely been settled, I~think I~would immediately drop dead again.
\end{quote}

As analogs of Fermat quotients, Wilson quotients, and Wilson primes, we introduce Lerch quotients and Lerch primes in Section~\ref{sec:L}, and Fermat-Wilson quotients and Wieferich-non-Wilson primes in Section~\ref{sec:W}. We define them by combining Fermat and Wilson quotients in apparently new ways.


\section{Lerch quotients and Lerch primes}
\label{sec:L}

In $1905$ Lerch~\cite{lerch} proved a congruence relating the Fermat and Wilson quotients of an odd prime.\\

\noindent{\bf Lerch's Formula.} {\it If a prime $p$ is odd, then}
\begin{equation*}
	\sum_{a=1}^{p-1} q_p(a) \equiv w_p \pmod{p}, 
\end{equation*}
{\it that is,}
\begin{align}
	\sum_{a=1}^{p-1}a^{p-1} - p - (p-1)! \equiv 0 \pmod{p^2}. \label{lf2}
\end{align}
\begin{proof}
 Replace $a$ with $ab$ in equation \eqref{EQ: qDef}. Substituting $a^{p-1}=pq_p(a)+1$ and $b^{p-1}=pq_p(b)+1,$ we deduce Eisenstein's logarithmic relation~\cite{eisenstein}
\begin{align*}
	q_p(ab) \equiv q_p(a) + q_p(b) \pmod{p}
\end{align*}
and Lerch's formula follows. For details, see \cite{lerch} or \cite{sm}.
\hfill $\Box$
\end{proof}

Ribenboim~\cite[p.~218]{rib00} explains the point of Lerch's formula this way:
\begin{quote}
Since the Fermat quotient is somehow hard to compute, it is more natural to relate their sums, over all the residue classes, to quantities defined by~$p.$
\end{quote}

Wilson quotients and Lerch's formula have been used (see \cite{sm}) to characterize solutions of the congruence
$$1^n + 2^n + \dotsb + k^n \equiv (k+1)^n \pmod{k^2}.$$

\subsection{Lerch quotients}
\label{subsec:Lquotients}

Lerch's formula allows us to introduce the Lerch quotient of an odd prime, by analogy with the classical Fermat and Wilson quotients of any prime.

\begin{definition} \label{DEF: lq}
The \emph{Lerch quotient} of an odd prime $p$ is the integer
\begin{align*}
	\ell_p :=&\ \frac{\sum_{a=1}^{p-1} q_p(a) - w_p}{p} =\ \frac{\sum_{a=1}^{p-1}a^{p-1} - p - (p-1)!}{p^2}. 
\end{align*}
\end{definition}

For instance,
\begin{align*}
	\ell_5 =&\ \frac{0+3+16+51-5}{5} = \frac{1+16+81+256-5-24}{25} = 13.
\end{align*}

The Lerch quotients of $p=3,5,7,11,13,17,19,23,29,\dotso$ are
\begin{align*}
 \ell_p =\ & 0, 13, 1356, 123229034, 79417031713, 97237045496594199,\\
& 166710337513971577670, 993090310179794898808058068,\\
& 60995221345838813484944512721637147449,\dotso,
\end{align*}
and for prime $p\le62563$ the only  Lerch quotient $ \ell_p$ that is itself a prime number is $\ell_5=13$ (see \cite[Sequence A197630]{oeis}). By contrast, the Wilson quotients $w_p$ of the primes $p= 5, 7, 11, 29, 773, 1321, 2621$ are themselves prime \cite[Section A2]{upint}, \cite[Sequence A050299]{oeis}.


\subsection{Lerch Primes and Bernoulli Numbers} \label{subsec:Lprimes}

We define Lerch primes by analogy with Wilson primes.

\begin{definition} \label{DEF: lp}
An odd prime $p$ is a \emph{Lerch prime} if $p$ divides $\ell_p,$ that is, if
\begin{align}
	\sum_{a=1}^{p-1} a^{p-1} - p - (p-1)! \equiv 0 \pmod{p^3}. \label{EQ: lpDef}
\end{align}
\end{definition}

For $p=3, 5, 7, 11, 13, 17, 19, 23, 29, 31, 37, 41, 43, 47, 53, 59, 61, 67, 71, 73, 79, 83,$ $89, 97,101, 103,\dotsc,$ we find that
\begin{eqnarray*}
\ell_p \equiv  &&0, 3, 5, 5, 6, 12, 13, 3, 7, 19, 2, 21, 34, 33, 52, 31, 51, 38, 32, 25, 25, 25, \\
&&53, 22, 98, 0,\dotsc \pmod{p}
\end{eqnarray*}
\cite[Sequence A197631]{oeis}, and so the first two Lerch primes are~$3$ and $103.$

We give a test for Lerch primes involving \emph{Bernoulli numbers}. Ubiquitous in number theory, analysis, and topology (see Dilcher~\cite{dilcher}), they are rational numbers~$B_n$ defined implicitly for $n\ge1$ by the symbolic recurrence relation$$(B+1)^{n+1}-B^{n+1}=0.$$(Ribenboim~\cite[p. 218]{rib00} says, ``Treat $B$ as an indeterminate and, after computing the polynomial in the left-hand side, replace $B^k$ by $B_k.$'') Thus for $n=1,$ we have $(B+1)^2-B^2=2B_1+1=0,$ and so $B_1=-1/2.$ Now with $n=2,$ we see that $(B+1)^3-B^3=3B_2+3B_1+1=0$ leads to $B_2=1/6.$ In this way, we get
\begin{equation*}
    B_3=0, B_4=-\frac{1}{30}, B_5=0, B_6=\frac{1}{42}, B_7=0, B_8=-\frac{1}{30}, B_9=0, B_{10}=\frac{5}{66}, \dotso.
\end{equation*}

In $1937$ (before the era of high-speed computers!) Emma Lehmer~\cite{lehmer} showed that $5$ and $13$ are the only Wilson primes $p\le211.$ To do this, she used her husband D.~H.~Lehmer's table of Bernoulli numbers up to $B_{220},$ together with \emph{Glaisher's congruence}~\cite{glaisher} (see also~\cite{lerch}), which holds for any prime $p$:
\begin{align}
    w_p &\equiv B_{p-1}+\frac1p -1 \pmod {p}. \label{wilsoncong}
\end{align}
Here recall the definition$$\frac ab\equiv 0\pmod{m} \quad \iff \quad m\mid a,\ \ \text{GCD}(a,b)=1.$$

For example, that $5$~is a Wilson prime, but $7$ is not, follows from the congruences
\begin{align*}
    w_5 \equiv B_4+\frac15 -1&=-\frac56 \equiv 0 \pmod {5},\\
    w_7 \equiv B_6+\frac17 -1&=-\frac56 \not\equiv 0 \pmod {7}.
\end{align*}

Multiplying Glaisher's congruence by $p$ and substituting $pw_p=(p-1)!+1$ yields \emph{E.~Lehmer's test: A prime $p$ is a Wilson prime if and only~if}
\begin{equation*}
	pB_{p-1}\equiv p-1 \pmod {p^2}. 
\end{equation*}

We provide an analogous test for Lerch primes.

\begin{theorem}[Test For Lerch Primes] A prime $p>3$ is a Lerch prime if and only if
\begin{equation}
	pB_{p-1} \equiv p+(p-1)! \pmod{p^3}. \label{lpcriterion2}
\end{equation}
\end{theorem}
\begin{proof} 
We first establish the following \textit{Criterion: an odd prime $p$ is a Lerch prime if and only if}
\begin{align}
	(B+p)^p \equiv p^2+p! \pmod{p^4}. \label{lpcriterion}
\end{align}
To see this, recall the classical application of Bernoulli numbers called \emph{Faulhaber's formula} (also known as \emph{Bernoulli's formula}\textemdash Knuth~\cite{knuth} has insights on this):
\begin{align}
	1^n+2^n+\dotsb+(k-1)^n &= \frac{(B+k)^{n+1}-B^{n+1}}{n+1}. \label{ff}
\end{align}
(See Conway and Guy~\cite[pp. 106--109]{cg} for a lucid proof.) Now set $k=p$ and $n=p-1$ in \eqref{ff}. It turns out that $B_p=0$ (indeed, $B_3=B_5=B_7=B_9=\dotsb=0;$ see~\cite[p.~109]{cg},~\cite[section~7.9]{hw}), and it follows that the congruences \eqref{EQ: lpDef} and \eqref{lpcriterion} are equivalent. This proves the Criterion.

To prove the Test, note that for any odd positive integer $p,$ the vanishing of $B_{2k+1}$ for $k\ge1$ implies
\begin{align}
	(B+p)^p  = p^p+p\cdot p^{p-1}B_1 + \sum_{k=1}^{(p-1)/2} \binom{p}{2k}p^{p-2k}B_{2k}. \label{expand(B+p)^p}
\end{align}
The von Staudt-Clausen theorem~\cite[p.~109]{cg},~\cite[section~7.9]{hw},~\cite[p.~340]{rib89} says in part that the denominator of $B_{2k}$ is the product of all primes~$q$ for which $(q-1)\mid 2k.$ (For instance, as $(2-1)\mid 2$ and $(3-1)\mid 2,$ the denominator of $B_2$ is $2\cdot3,$ agreeing with $B_2=1/6.$) Thus, if $p$ is an odd prime, then on the right-hand side of~\eqref{expand(B+p)^p} only $B_{p-1}$ has denominator divisible by $p.$ From this we see, for $p\ge5,$ that~$p^4$ divides the numerator of each term except $p^2 B_{p-1}.$ (For the $k=(p-3)/2$ term, this uses $p\mid \binom{p}{p-3}.$) Therefore, the congruence
\begin{align}
	(B+p)^p  \equiv p^2 B_{p-1} \pmod{p^4} \label{p^2 B_{p-1}}
\end{align}
holds for \emph{all} primes $p>3.$ Substituting \eqref{p^2 B_{p-1}} into Criterion~\eqref{lpcriterion} and dividing by $p,$ we arrive at Test~\eqref{lpcriterion2}.\hfill $\Box$
\end{proof}

As a bonus, \eqref{p^2 B_{p-1}} affords a proof of Glaisher's congruence.

\begin{corollary} \label{COR:glaisher}
The congruence~\eqref{wilsoncong} holds. Equivalently, if $p$ is \emph{any} prime, then
\begin{equation}
    pB_{p-1} \equiv p+(p-1)!  \pmod {p^2}. \label{wilsoncong2}
\end{equation}
\end{corollary}
\begin{proof} To see the equivalence, substitute~\eqref{EQ: wDef} into~\eqref{wilsoncong} and multiply by $p.$ To prove~\eqref{wilsoncong2}, first verify it for $p=2$ and $3.$ If $p>3,$ use~\eqref{lf2}, \eqref{ff}, and the fact that $B_p=0$ to get \mbox{$(B+p)^p \equiv p^2+p! \pmod{p^3}$}. Then \eqref{p^2 B_{p-1}} and division by $p$ yield~\eqref{wilsoncong2}.\hfill $\Box$
\end{proof}

Notice that the congruences \eqref{lpcriterion2} and \eqref{wilsoncong2} are the same, except that in \eqref{lpcriterion2} the modulus is~$p^3,$ while in \eqref{wilsoncong2} it is $p^2.$ However, one cannot prove Corollary~\ref{COR:glaisher} trivially (by reducing \eqref{lpcriterion2} modulo $p^2$ instead of~$p^3$), because \eqref{lpcriterion2} holds only for Lerch primes, whereas \eqref{wilsoncong2} holds for all primes.


\subsection{Computing Lerch primes: a surprising crossover} \label{subsec:Lcomp}

Let us compare two methods of computing Lerch primes: Definition~\eqref{EQ: lpDef} and Test~\eqref{lpcriterion2}. Both require, essentially, computation modulo~$p^3.$ The Test seems simpler, but on the other hand it requires computing $B_{p-1}$ modulo $p^2.$

To find out which is faster, we used the code
\begin{quote}
\verb|If[Mod[Sum[PowerMod[a,p-1,p^3], {a,1,p-1}] - p - (p-1)!, p^3] |\\
\verb|   == 0, Print[p]]|
\end{quote}
\noindent in a \emph{Mathematica} (version 7.0.0) program for \eqref{EQ: lpDef}, and we used the code
\begin{quote}
\verb|If[Mod[Numerator[p*Mod[BernoulliB[p-1],p^2] - p - (p-1)!], p^3]|\\
\verb|   == 0, Print[p]]|
\end{quote}
\noindent in a program for~\eqref{lpcriterion2}. Here \verb|Mod[a,m]| gives $a\ \text{mod}\ m,$ \verb|PowerMod[a,b,m]| gives $a^{b}\ \text{mod}\ m$ (and is faster than \verb|Mod[a^b,m]|), and \verb|BernoulliB[k]| gives $B_k.$

Table~1 shows the CPU time (on a MacBook Air computer with OS~X 10.6 and 2.13GHz Intel processor) for each program to decide whether $p$ is a Lerch prime.

Note the surprising crossover in the interval \mbox{$10007\le p\le20011$}: before it, Test~\eqref{lpcriterion2} is much faster than Definition~\eqref{EQ: lpDef}, but after the interval the reverse is true. Notice also that for $p>10^4$ the CPU times of~\eqref{EQ: lpDef} grow at about the same rate as~$p,$ while those of~\eqref{lpcriterion2} balloon at more than double that rate.

The programs for \eqref{EQ: lpDef} and \eqref{lpcriterion2} searched up to $10^4$ in about $47.3$ and $0.6$ seconds, respectively, and found the Lerch primes $3,103,839,$ and $2237$ (see \cite[Sequence A197632]{oeis}). There are no others up to $10^6$, by the program for \eqref{EQ: lpDef}, which consumed about $160$ hours. (To run the program for \eqref{lpcriterion2} that far up was not feasible.)

Marek Wolf, using a modification of \eqref{EQ: lpDef}, has computed that there are no Lerch primes in the intervals $1000003\le p\le4496113$ and $18816869\le p \le 18977773$, as well as $32452867 \le p\le32602373$. His computation took six months of CPU time on a 64-bit AMD Opteron 2700 MHz processor at the cluster \cite{klaster}.

\vspace{1em}
\begin{center}
\begin{tabular}{r| c c r@{.} c}
\hline
&\multicolumn{4}{c}{\, CPU time in seconds} \\
$p$\hspace{.4cm} & Definition & vs. & \multicolumn{2}{c}{\hspace{.4cm}Test}  \\
\hline
$5$ & 0.000052 &  $>$  &0&000040\\
$11$ & 0.000069 & $>$ & 0&000044\\
$101$ & 0.000275 & $>$ & 0&000064\\
$1009$ & 0.002636 & $>$ &0&000156\\
$10007$ & 0.088889 &$>$ & 0&002733\\
\hline
$20011$ & 0.183722 &$<$ & 0&337514\\
$30011$ & 0.294120 & $<$ &0&816416\\
$100003$ & 1.011050 & $<$ &10&477100\\
$200003$ & 2.117640 & $<$ &49&372000\\
$300007$ & 3.574630 & $<$ &121&383000\\
$1000003$ & 12.647500\ \  &  $<$ &1373&750000\\
\hline
\end{tabular}
\end{center}
\begin{center}
Table 1: Time each of two programs takes to compute whether $p$ is a Lerch prime.
\end{center}


\subsection{Generalizations} \label{subsec: eulergauss}

Euler and Gauss extended Fermat's little theorem and Wilson's theorem, respectively, to congruences with a composite modulus~$n$---see~\cite[Theorems~71 and~129]{hw}. The corresponding generalizations of Fermat and Wilson quotients and Wilson primes are called {\em Euler quotients}~$q_n(a),$ {\em generalized Wilson quotients}~$w_n,$ and {\em Wilson numbers}~$n\mid w_n$ (see~\cite[sequences A157249 and A157250]{oeis}). (The $w_n$ are not called ``Gauss quotients;'' that term appears in the theory of hypergeometric functions.)  In $1998$ Agoh, Dilcher, and Skula \cite[Proposition~2.1]{ads} (see also Dobson \cite{dobson} and Cosgrave and Dilcher \cite{cd}) extended Lerch's formula to a congruence between the $q_n(a)$ and $w_n.$

Armed with these facts, one can define {\em generalized Lerch quotients}~$\ell_n$ and {\em Lerch numbers}~$n\mid\ell_n.$ But that's another story for another time.

\subsection{Open Problems} \label{subsec:Lprob}

To conclude this section, we pose some open problems.

\begin{prob}
Is $\ell_5=13$ the only prime Lerch quotient?
\end{prob}

\begin{prob}
Is there a fifth Lerch prime? Are there infinitely many?
\end{prob}


Of the $78498$ primes $p<10^6,$ only four are Lerch primes. Thus the answer to the next question is clearly yes; the only thing lacking is a proof!

\begin{prob} \label{Q: nonlerch}
Do infinitely many \emph{non}-Lerch primes exist?
\end{prob}

As the known Lerch primes $3,103, 839,2237$ are distinct from the known Wilson primes $5,13,563,$ we may ask:

\begin{prob}
Is it possible for a number to be a Lerch prime and a Wilson prime simultaneously?
\end{prob}

Denoting the $n$th prime by $p_n,$ the known Wilson primes are $p_3,p_6,p_{103}.$ The primes among the indices $3,6,103,$ namely, $3$ and $103,$ are Lerch primes. This leads to the question:

\begin{prob} \label{Q: coincidence}
If $p_n$ is a Wilson prime and $n$ is prime, must $n$~be a Lerch prime?
\end{prob}

The answer to the converse question---if $n$ is a Lerch prime, must $p_n$ be a Wilson prime?---is no: $p_{839}$ and $p_{2237}$ lie strictly between $563$ and $5\times 10^{8},$ where according to \cite{cdp} there are no Wilson primes.

In connection with Problem \ref{Q: coincidence}, compare Davis's ``Are there coincidences in mathematics?''~\cite{davis} and Guy's ``The strong law of small numbers''~\cite{law}.


\section{Fermat-Wilson quotients and the WW primes 2, 3, 14771} \label{sec:W}

Suppose that a prime $p$~is not a Wilson prime, so that $p$~does not divide its Wilson quotient $w_{\mspace{1mu}p}.$ Then in the Fermat quotient $q_{\mspace{1mu}p}(a)$ of~$p$ base $a$, we may take $a=w_p.$

\begin{definition} \label{DEF: fw}
If $p$~is a non-Wilson prime, then the \emph{Fermat-Wilson quotient of~$p$} is the integer
\begin{equation*}
     q_{\mspace{1mu}p}(w_{\,p})=\frac{w_{\,p}^{\,p-1}-1}{p}.
\end{equation*}
\end{definition}
For short we write$$g_{\,p}:=q_{\,p}(w_{\,p}).$$

The first five non-Wilson primes are $2,3,7,11,17$. Since $w_{\,2}=w_{\,3}=1,$ $w_{\,7}=103,$ and $w_{\,11}=329891,$ the first four Fermat-Wilson quotients are $g_{\,2}=\,g_{\,3}=0,$
\begin{eqnarray*}
     g_{\,7}&=&\frac{103^{\,6}-1}{7} = 170578899504,
\end{eqnarray*}
and
\begin{eqnarray*}
     g_{\,11}&=&\frac{329891^{\,10}-1}{11}\\
               &=& 1387752405580695978098914368989316131852701063520729400
\end{eqnarray*}
\cite[Sequence A197633]{oeis}. The fifth one, $g_{\,17},$ is a $193$-digit number.


\subsection{The GCD of all Fermat-Wilson quotients} \label{subsec:GCD}

We saw that at least one Lerch quotient and seven Wilson quotients are prime numbers. What about Fermat-Wilson quotients?

\begin{theorem} \label{THM:gcd}
The greatest common divisor of all Fermat-Wilson quotients is~$24$. In particular, $q_{\,p}(w_{\,p})$ is never prime.
\end{theorem}
\begin{proof}
The prime factorizations of $q_{\,p}(w_{\,p})=g_{\,p}$ for $p=7$ and $11$ are
\begin{eqnarray*}
     g_{\,7} = 2^{\,4}\cdot3^{\,2}\cdot13\cdot17\cdot19\cdot79\cdot3571
\end{eqnarray*}
and
\begin{eqnarray*}
     g_{\,11}&=& 2^{\,3}\cdot3\cdot5^{\,2}\cdot37\cdot61\cdot71\cdot271\cdot743\cdot2999\cdot89671\cdot44876831\\
               &\quad&\cdot\ 743417279981\cdot7989680529881.
\end{eqnarray*}
Since $g_{\,2}=\,g_{\,3}=0,$ we thus have$$\text{GCD}(g_{\,2},g_{\,3},g_{\,7},g_{\,11})=2^{\,3}\cdot3=24.$$

To complete the proof, we show that $24$ divides $g_{\,p}$ whenever $p>3.$ Since$$p\,w_{\,p}=(p-1)!+1,$$it is clear that if $p\ge5,$ then $p\,w_{\,p},$ and hence $w_{\,p},$ is not divisible by $2$ or $3.$ As even powers of such numbers are $\equiv 1 \!\pmod{8}$ and $\equiv 1\! \pmod{3},$ and so $\equiv 1\! \pmod{24},$ it follows that $p\, g_{\,p} \ (= w_{\,p}^{\,p-1}-1),$ and hence $g_{\,p},$ is divisible by $24.$
\hfill $\Box$
\end{proof}

\subsection{Wieferich primes base a}  \label{subsec:Wief}

Given an integer $a,$ a prime $p$ is called a \emph{Wieferich prime base $a$} if the supercongruence
\begin{equation} 
    a^{\,p-1} \equiv 1 \pmod{p^{\,2}} \label{wief}
\end{equation}
holds. For instance, $11$  is a Wieferich prime base~$3,$ because$$3^{10}-1= 59048= 11^2\cdot 488.$$

Paraphrasing Ribenboim \cite[p. 264]{rib89}, it should be noted that, contrary to the congruence $a^{\,p-1} \equiv 1 \pmod{p}$ which is satisfied by every prime $p$ not dividing $a,$ the Wieferich supercongruence \eqref{wief} is very rarely satisfied.

When it is, $p$~cannot divide~$a,$ and so the Fermat quotient $q_{\,p}(a)$ is an integer. In fact, \eqref{EQ: qDef} shows that a prime $p$~is a Wieferich prime base~$a$ if and only if $p$~does not divide~$a$ but does divide $q_{\,p}(a)$.

In $1909,$ while still a graduate student at the University of M\"{u}nster in Germany, Wieferich created a sensation with a result related to Fermat's Last Theorem: \emph{If $x^{\,p} + y^{\,p} = z^{\,p},$ where $p$ is an odd prime not dividing any of the integers $x, y,$ or $z,$ then $p$~is a Wieferich prime base~2}. One year later, Mirimanoff proved that \emph{$p$~is also a Wieferich prime base~3}. (See \cite[pp. 110-111]{dickson}, \cite[Chapter 8]{rib00}, and \cite[p. 163]{wells}.)

The only known Wieferich primes base $2$ (also simply called \emph{Wieferich primes}) are $1093$ and $3511,$ discovered by Meissner in 1913 and Beeger in 1922, respectively. In $2011$ Dorais and Klyve \cite{dk} computed that there are no others up to $6.7\times10^{15}$. It is unknown whether infinitely many exist. (Neither is it known whether there are infinitely many \emph{non}-Wieferich primes base~$2$. However, Silverman has proved it assuming the $abc$-conjecture---see his pleasantly-written paper \cite{silverman}.) Likewise, only two Wieferich primes base $3$ (also known as \emph{Mirimanoff primes}) have been found, namely, $11$ and $1006003$. The second one was uncovered by Kloss in 1965. An unanswered question is whether it is possible for a number to be a Wieferich prime base~$2$ and base~$3$ simultaneously. (See \cite[section A3]{upint} and \cite[pp. 263--276, 333--334]{rib89}.)

For tables of all Wieferich primes $p$ base $a$ with $2<p<2^{\,32}$ and $2\le a\le99,$ see Montgomery \cite{montgomery}.

\subsection{The Wieferich-non-Wilson primes 2, 3, 14771}  \label{subsec:WW}

Let us consider Wieferich primes $p$ base $a$ where $a$ is the Wilson quotient of $p.$

\begin{definition} \label{DEF: ww}
Let $p$ be a non-Wilson prime, so that its Fermat-Wilson quotient $q_{\,p}(w_{\,p})$ is an integer. If $p$ divides $q_{\,p}(w_{\,p})$---equivalently, if the supercongruence
\begin{equation}
    w_{\,p}^{\,p-1} \equiv 1 \pmod{p^{\,2}} \label{wcong}
\end{equation}
holds---then $p$ is a Wieferich prime base $w_p,$ by definition \eqref{wief}. In that case, we call~$p$ a \emph{Wieferich-non-Wilson prime}, or \emph{WW prime} for short.
\end{definition}

For the non-Wilson primes $p=2, 3, 7, 11, 17, 19, 23, 29, 31, 37, 41, 43, 47, 53, 59,$ $61,67, 71, 73, 79, 83,\dotsc,$ the Fermat-Wilson quotients $q_{\,p}(w_{\,p})=g_{\,p}$ are congruent modulo~$p$ to
\begin{eqnarray*}
g_p \equiv  0, 0, 6, 7, 9, 7, 1, 6, 18, 17, 30, 11, 25, 30, 24, 46, 64, 16, 18, 4, 29,\dotsc \pmod{p}
\end{eqnarray*}
\cite[Sequence A197634]{oeis}. In particular, $2$ and $3$ are WW primes. But they are trivially so, because $g_{\,2}$ and $g_{\,3}$ are \emph{equal} to zero.

Is there a ``non-trivial'' WW prime? Perhaps surprisingly, the answer is yes but the smallest one is $14771$. In the next subsection, we give some details on using a computer to show that $14771$ is a WW prime. It is ``non-trivial'' because $ g_{\,14771}\neq0.$ In fact, taking logarithms, one finds that
\begin{align*}
     g_{\,14771}=\frac{\left(\frac{14770!+1}{14771}\right)^{\,14770} -1}{14771} >10^{\,8\times10^{\,8}},
\end{align*}
so that the number $g_{\,14771}$ has more than $800$ million decimal digits.


\subsection{Computer search} \label{subsec:search}

To search for WW primes, one can use a computer to calculate whether or not a given prime~$p$ satisfies condition \eqref{wcong}. Explicitly, if the number
\begin{equation}
    \left(\frac{(p-1)!+1}{p}\right)^{p-1} \text{ mod } p^{\,2} \label{explw(p)}
\end{equation}
is equal to $1$, then $p$ is a WW prime.

{\it Mathematica}'s function Mod$[a,m]$ can compute \eqref{explw(p)} when $p$ is small. But if $p$ is large, an ``Overflow'' message results. However, it is easy to see that in \eqref{explw(p)} one may replace $(p-1)!$ with $(p-1)!$ mod $p^{\,3},$ a much smaller number.

For example, it takes just a few seconds for a program using the code
\begin{quote}
\verb|If[PowerMod[(Mod[(p-1)!, p^3] + 1)/p, p-1, p^2] == 1, Print[p]]|
\end{quote}
to test the first $2000$ primes and print the WW primes $2,3,14771$ (see \cite[Sequence A197635]{oeis}).

Michael Mossinghoff, employing the GMP library~\cite{gmp}, has computed that there are no other WW primes up to $10^7$.


\subsection{More open problems} \label{subsec:moreprobs}

We conclude with three more open problems.
\begin{prob}
Can one prove that $14771$ is a WW prime (i.e., that $14771$ divides $ g_{\,14771}$) without using a computer?
\end{prob}

Such a proof would be analogous to those given by Landau and Beeger that $1093$ and $3511,$ respectively, are Wieferich primes base~$2.$ (See Theorem~91 and the notes on Chapter VI in \cite{hw}, and ``History and search status'' in \cite{wiki}.) However, proofs for Wieferich primes are comparatively easy, because (high) powers are easy to calculate in modular arithmetic, whereas factorials are unlikely to be calculable in logarithmic time.
\begin{prob}
Is there a fourth WW prime? Are there infinitely many?
\end{prob}

Comments similar to those preceding Problem \ref{Q: nonlerch} also apply to the next question.

\begin{prob} \label{Prob:nonWW}
Do infinitely many \emph{non}-WW primes exist?
\end{prob}

Is it possible to solve Problem~\ref{Q: nonlerch} or Problem~\ref{Prob:nonWW} assuming the $abc$-conjecture? (See the remark in Section~\ref{subsec:Wief} about Silverman's proof.)


\section*{Acknowledgments}
I am grateful to Wadim Zudilin for suggestions on the Test, for a simplification in computing WW primes, and for verifying that there are no new ones up to $30000$, using PARI/GP \cite{pari}. I thank Marek Wolf for computing Lerch primes, and Michael Mossinghoff for computing WW primes.


\end{document}